\documentclass{amsart}

\usepackage{enumerate}
\usepackage{amsmath}%
\usepackage{amsfonts}%
\usepackage{amssymb}%
\usepackage{graphicx}
\usepackage{mathrsfs}
\newtheorem{theorem}{Theorem}
\theoremstyle{plain}

\newtheorem{corollary}[theorem]{Corollary}

\newtheorem{lemma}[theorem]{Lemma}

\newtheorem{proposition}[theorem]{Proposition}

\numberwithin{equation}{section}
\numberwithin{theorem}{section}
\numberwithin{case}{section}

\numberwithin{subcase}{case}

\def \a{\alpha}
\def \e{\epsilon}
\def \r{\gamma}

\begin{document}
\title{Near Perfect Matchings in $k$-uniform Hypergraphs}
\author{Jie Han}
\thanks{Corresponding author: Jie Han, Georgia State University, Atlanta, GA, USA. Email: jhan22@gsu.edu.}
\date{\today}

\maketitle

\begin{abstract}
Let $H$ be a $k$-uniform hypergraph on $n$ vertices where $n$ is a sufficiently large integer not divisible by $k$. We prove that if the minimum $(k-1)$-degree of $H$ is at least $\lfloor n/k \rfloor$, then $H$ contains a matching with $\lfloor n/k\rfloor$ edges. This confirms a conjecture of R\"odl, Ruci\'nski and Szemer\'edi \cite{RRS09}, who proved that minimum $(k-1)$-degree $n/k+O(\log n)$ suffices. More generally, we show that $H$ contains a matching of size $d$ if its minimum codegree is $d<n/k$, which is also best possible.
\end{abstract}

\section{Introduction}

Given $k\ge 2$, a $k$-uniform hypergraph (in short, \emph{$k$-graph}) consists of a vertex set $V(H)$ and an edge set $E(H)\subseteq \binom{V(H)}{k}$, where every edge is a $k$-element subset of $V(H)$. A \emph{matching} in $H$ is a collection of vertex-disjoint edges of $H$. A \emph{perfect matching} $M$ in $H$ is a matching that covers all vertices of $H$. Clearly a perfect matching in $H$ exists only if $k$ divides $|V(H)|$. When $k$ does not divide $n=|V(H)|$, we call a matching $M$ in $H$ a \emph{near perfect matching} if $|M|=\lfloor n/k \rfloor$.

Given a $k$-graph $H$ with a set $S$ of $d$ vertices (where $1 \le d \le k-1$) we define $\deg_{H} (S)$ to be the number of edges containing $S$ (the subscript $H$ is omitted if it is clear from the context). The \emph{minimum $d$-degree $\delta _{d} (H)$} of $H$ is the minimum of $\deg_{H} (S)$ over all $d$-vertex sets $S$ in $H$.  
We refer to $\delta _{k-1} (H)$ as the \emph{minimum codegree} of $H$.

Over the last few years there has been a strong focus in establishing minimum $d$-degree thresholds that force a perfect matching in a $k$-graph \cite{AFHRRS, CzKa, HPS, Khan2, Khan1, KO06mat, KOT, MaRu, Pik, RR, RRS06mat, RRS09, TrZh13}. In particular, R\"odl, Ruci\'nski and Szemer\'edi \cite{RRS09} determined the minimum codegree threshold that ensures a perfect matching in a $k$-graph on $n$ vertices for all $k\ge 3$ and sufficiently large $n\in k\mathbb{N}$. The threshold is $\frac{n}{2}-k+C$, where $C\in\{3/2, 2, 5/2, 3\}$ depends on the values of $n$ and $k$. 
In contrast, they proved that the minimum codegree threshold that ensures a near perfect matching in a $k$-graph on $n\notin k\mathbb{N}$ vertices is between $\lfloor \frac{n}{k}\rfloor$ and $\frac{n}{k}+O(\log n)$. It is conjectured, in \cite{RRS09} and \cite[Problem 3.3]{RR}, that this threshold is $\lfloor \frac{n}{k}\rfloor$. In this note we verify this conjecture.

\begin{theorem}\label{thm:main}
For any integer $k\ge 3$, let $n$ be a sufficiently large integer which is not divisible by $k$. Suppose $H$ is a $k$-uniform hypergraph on $n$ vertices with $\delta_{k-1}(H)\ge \lfloor\frac{n}{k}\rfloor$. Then $H$ contains a matching of size $\lfloor\frac{n}{k}\rfloor$.
\end{theorem}

It is also natural to ask for the minimum codegree threshold for the \emph{matching number} of $k$-graphs, namely, the size of a maximum matching. The following theorem \cite[Fact 2.1]{RRS09} is obtained by a greedy algorithm. Let $\nu(H)$ be the size of a maximum matching in $H$.
\begin{theorem}\cite{RRS09}\label{fact:RRS09}
Let $n\ge k\ge 2$. For every $k$-uniform hypergraph $H$ on $n$ vertices,
\[
\nu(H)\ge  \delta_{k-1}(H) \text{ if } \delta_{k-1}(H) \le \left\lfloor \frac{n}{k} \right\rfloor -k+2.
\]
\end{theorem}

Note that for $n\in k\mathbb{N}$ and $\frac{n}{k}\le \delta_{k-1}(H) \le \frac{n}{2}-k$, $H$ may not contain a perfect matching, namely, a matching of size $\frac{n}{k}$ (see \cite{RRS09}). So the only open cases are when $\left\lfloor \frac{n}{k} \right\rfloor -k+3\le \delta_{k-1}(H)< \frac{n}{k}$. In this note, we close this gap for large $n$.
\begin{corollary}\label{thm:mat_num}
For any integer $k\ge 3$, let $n$ be a sufficiently large integer. For every $k$-uniform hypergraph $H$ on $n$ vertices,
\[
\nu(H)\ge  \delta_{k-1}(H) \text{ if } \delta_{k-1}(H)< \frac{n}{k}.
\]
\end{corollary}

\begin{proof}
Let $\delta_{k-1}(H)=\left\lfloor \frac{n}{k} \right\rfloor -c$. We only prove Corollary \ref{thm:mat_num} in the cases when $1\le c\le k-3$, since Theorem \ref{fact:RRS09} covers the cases when $c\ge k-2$ and Theorem \ref{thm:main} covers the case when $\delta_{k-1}(H)=\left\lfloor \frac{n}{k} \right\rfloor<\frac{n}{k}$. 
Let $r\equiv n \mod k$ such that $0\le r\le k-1$. Note that $\left\lfloor \frac{n}{k} \right\rfloor = \left\lfloor \frac{n+c}{k} \right\rfloor$ if $r+c<k$ and $\left\lfloor \frac{n}{k} \right\rfloor+1 = \left\lfloor \frac{n+c+1}{k} \right\rfloor$ otherwise.
For the first case, we add $c$ vertices to $H$ and get $H'$ such that $H'$ contains all edges of $H$ and all $k$-sets containing any of these new vertices. Note that $H'$ has $n+c$ vertices and $\delta_{k-1}(H')=\left\lfloor \frac{n+c}{k} \right\rfloor$. Moreover, $k$ does not divide $n+c$ since $1\le r+c<k$. We apply Theorem \ref{thm:main} on $H'$ and get a near perfect matching $M$ of $H'$. Deleting up to $c$ edges from $M$ that contain the new vertices, we get a matching in $H$ of size $\left\lfloor \frac{n}{k} \right\rfloor -c$.

In the second case, we add $c+1$ vertices to $H$ and get $H'$ such that $H'$ contains all edges of $H$ and all $k$-sets containing any of these new vertices. Note that $H'$ has $n+c+1$ vertices and $\delta_{k-1}(H')=\left\lfloor \frac{n}{k} \right\rfloor+1=\left\lfloor \frac{n+c+1}{k} \right\rfloor$. Moreover, $k$ does not divide $n+c+1$ since $k+1\le r+c+1\le 2k-3$. Similarly we apply Theorem \ref{thm:main} on $H'$ and get a near perfect matching $M$ of $H'$. Deleting up to $c+1$ edges from $M$ that contain the new vertices, we get a matching in $H$ of size $\left\lfloor \frac{n}{k} \right\rfloor+1 - (c+1) =\left\lfloor \frac{n}{k} \right\rfloor -c$.
\end{proof}

It is easy to see that Theorem \ref{thm:main} and Corollary \ref{thm:mat_num} are best possible. For an integer $0\le d<\frac{n}{k}$, let $H$ be a $k$-graph with a partition $A\cup B$ of the vertex set $V(H)$ such that $|A|=d$ and $E(H)$ consists of all $k$-tuples that intersect $A$. Since every edge intersects $A$, we have $\nu(H)=\delta_{k-1}(H)=|A|=d$.

Let us describe this interesting phenomenon by the following dynamic process. Consider a $k$-graph $H$ on $n$ vertices with $E(H)=\emptyset$ at the beginning and add edges to $E(H)$ gradually. Corollary \ref{thm:mat_num} says $\nu(H)\ge \delta_{k-1}(H)$ when $\delta_{k-1}(H) < \frac{n}{k}$. In order to guarantee a perfect matching, $\delta_{k-1}(H)$ needs to be about $n/2$ \cite{RRS09}.

\medskip
As a typical approach to obtain exact results, our proof of  Theorem~\ref{thm:main} consists of an \emph{extremal case} and a \emph{nonextremal case}.
We say that $H$ is \emph{$\r$-extremal} if $V(H)$ contains an independent subset $B$ of order at least $(1-\r)\frac{k-1}k n$.

\begin{theorem}[Nonextremal case]\label{thm:next}
For any integer $k\ge 3$ and constant $\gamma>0$, there is an integer $n_0$ such that the following holds. Let $n\ge n_0$ be an integer not divisible by $k$ and let $H$ be an $n$-vertex $k$-graph with $\delta_{k-1}(H)\ge \frac{n}{k} - \r n$. If $H$ is not $5k\r$-extremal, then $H$ contains a near perfect matching.
\end{theorem}

\begin{theorem}[Extremal case]\label{thm:ext}
For any integer $k\ge 3$, there exist an $\e>0$ and an integer $n_1$ such that the following holds. Let $n\ge n_1$ be an integer not divisible by $k$ and let $H$ be an $n$-vertex $k$-graph with $\delta_{k-1}(H)\ge \lfloor\frac{n}{k}\rfloor$. If $H$ is $\e$-extremal, then $H$ contains a near perfect matching.
\end{theorem}

Theorem \ref{thm:main} follows from Theorem \ref{thm:next} and Theorem \ref{thm:ext} immediately.

We prove Theorem \ref{thm:next} by the absorbing method, initiated  by R\"odl, Ruci\'nski and Szemer\'edi \cite{RRS06}. Given a set $S$ of $k+1$ vertices, we call an edge $e\in E(H)$ disjoint from $S$ \emph{$S$-absorbing} if there are two disjoint edges $e_1$ and $e_2$ in $E(H)$ such that $|e_1\cap S|=k-1$, $|e_1\cap e|=1$, $|e_2\cap S|=2$, and $|e_2\cap e|=k-2$. 
Note that this is not the absorbing in the usual sense because $e_1\cup e_2$ misses one vertex of $S\cup e$.
Let us explain how such absorbing works. Let $S$ be a $(k+1)$-set and $M$ be a matching, where $V(M)\cap S=\emptyset$, which contains an $S$-absorbing edge $e$. Then $M$ can ``absorb'' $S$ by replacing $e$ in $M$ by $e_1$ and $e_2$ (one vertex of $e$ becomes uncovered).
The following absorbing lemma was proved in \cite[Fact 2.3]{RRS09} with the conclusion that \emph{the number of $S$-absorbing edges in $M$ is at least $k-2$}. However, its proof shows that $k-2$ can be replaced by any constant. Note that we do not require that $k$ does not divide $n$ in Lemma \ref{lem:abs} and Lemma \ref{lem:alm_mat}.

\begin{lemma}\cite[Absorbing lemma]{RRS09}\label{lem:abs}
For all $c, \r>0$ there exist $C>0$ and $n_2$ such that if $H$ is a $k$-graph with $n\ge n_2$ vertices and $\delta_{k-1}(H)\ge c n$, then there exists a matching $M'$ in $H$ of size $|M'|\le C \log n$ and such that for every $(k+1)$-tuple $S$ of vertices of $H$, the number of $S$-absorbing edges in $M'$ is at least $k/\r$.
\end{lemma}

We also need the following lemma, which provides a matching that covers all but a constant number of vertices when $H$ is not extremal. 

\begin{lemma}[Almost perfect matching]\label{lem:alm_mat}
For any integer $k\ge 3$ and constant $\gamma>0$ the following holds. Let $H$ be an $n$-vertex $k$-graph such that $n$ is sufficiently large and $\delta_{k-1}(H)\ge \frac{n}{k}-\r n$. If $H$ is not $2k\r$-extremal, then $H$ contains a matching that covers all but at most $k^2/\r$ vertices.
\end{lemma}

Now let us compare our proof with the proof in \cite{RRS09}, which showed that $\delta_{k-1}(H)\ge \frac{n}{k}+O(\log n)$ guarantees a near perfect matching. In \cite{RRS09}, the authors first build an absorbing matching of size $C\log n$ and then apply Theorem \ref{fact:RRS09} in the remaining $k$-graph. Finally, they absorb the leftover vertices and get the near perfect matching.
In our proof, instead of Theorem \ref{fact:RRS09}, we apply Lemma \ref{lem:alm_mat} after building the absorbing matching.
Lemma \ref{lem:alm_mat} only requires a weaker degree condition $\delta_{k-1}(H)\ge \frac{n}{k}-\r n$ and the condition that $H$ is not extremal. We then handle the extremal case separately.

\section{Proof of Theorem \ref{thm:next}}

In this section we prove Theorem \ref{thm:next} with the help of Lemma \ref{lem:abs} and Lemma \ref{lem:alm_mat}.

\begin{proof}[Proof of Lemma \ref{lem:alm_mat}]
Let $M=\{e_1, e_2, \dots, e_m\}$ be a maximum matching of size $m$ in $H$. Let $V'$ be the set of vertices covered by $M$ and let $U$ be the set of vertices which are not covered by $M$. We assume that $H$ is not $2k\r$-extremal and $|U|> k^2/\r$. Note that $U$ is an independent set by the maximality of $M$. We arbitrarily partition all but at most $k-2$ vertices of $U$ as disjoint $(k-1)$-sets $A_1, \dots, A_t$ where $t=\lfloor \frac{|U|}{k-1}\rfloor>\frac{k}{\r}$.

Let $D$ be the set of vertices $v\in V'$ such that $\{v\}\cup A_i\in E(H)$ for at least $k$ sets $A_i$, $i\in [t]$. 
We claim that $|e_i\cap D|\le 1$ for any $i\in [m]$. Otherwise, assume that $x, y\in e_i\cap D$. By the definition of $D$, we can  pick $A_i, A_j$ for some distinct $i, j\in [t]$ such that $\{x\}\cup A_i\in E(H)$ and $\{y\}\cup A_j\in E(H)$. We obtain a matching of size $m+1$ by replacing $e_i$ in $M$ by $\{x\}\cup A_i$ and $\{y\}\cup A_j$, contradicting the maximality of $M$.

Next we show that $|D|\ge (\frac 1k -2\r)n$. By the minimum degree condition, we have
\[
t\left(\frac 1k -\r \right)n\le \sum_{i=1}^t \deg(A_i)\le |D| t +n\cdot k,
\]
where we use the fact that $U$ is an independent set.
So we get
\[
|D|\ge \left(\frac 1k -\r \right)n - \frac{n k}{t} > \left(\frac 1k -2\r \right)n,
\]
where we use $t>k/\r$.

Let $V_D:=\bigcup\{e_i, e_i\cap D\neq \emptyset\}$. Note that $|V_D\setminus D|=(k-1)|D|\ge (k-1)(\frac 1k -2\r)n$. Since $H$ is not $2k \r$-extremal, $H[V_D\setminus D]$ contains at least one edge, denoted by $e_0$.
We assume that $e_0$ intersects $e_{i_1}, \dots, e_{i_l}$ in $M$ for some $2\le l\le k$. Suppose $\{v_{i_j}\}= e_{i_j}\cap D$ for all $j\in [l]$. By the definition of $D$, we can greedily pick $A_{i_1}, \dots, A_{i_l}$ such that  $\{v_{i_j}\}\cup A_{i_j} \in E(H)$ for all $j\in [l]$. Let $M''$ be the matching obtained from replacing the edges $e_{i_1}, \dots, e_{i_l}$ by $e_0$ and $\{v_{i_j}\}\cup A_{i_j}$ for $j\in [l]$. Thus, $M''$ has $m+1$ edges, contradicting the maximality of $M$.
\end{proof}

Now we prove Theorem \ref{thm:next}.

\begin{proof}[Proof of Theorem \ref{thm:next}]
Suppose $H$ is a $k$-graph on $n\notin k\mathbb{N}$ vertices with $\delta_{k-1}(H)\ge n/k - \r n$ and $H$ is not $5k\r$-extremal. In particular, $\r<\frac{1}{5k}$. Since $\delta_{k-1}(H)\ge \frac{n}{2k}$, we first apply Lemma \ref{lem:abs} on $H$ with $c=\frac{1}{2k}$ and find the absorbing matching $M'$ of size at most $C\log n$ such that for every set $S$ of $k+1$ vertices of $H$, the number of $S$-absorbing edges in $M'$ is at least $k/\r$.

Let $H'=H[V(H)\setminus V(M')]$ and $n'=|V(H')|$. Note that $\delta_{k-1}(H')\ge \delta_{k-1}(H) - k C\log n > (\frac1{k} - 2\r) n'$. If $H'$ is $4k\r$-extremal, namely, $V(H')$ contains an independent set $B$ of order at least $(1 -4k\r)\frac{k-1}{k}n'$, then since
\[
(1 -4k\r)\frac{k-1}{k}n' \ge (1 -5k\r)\frac{k-1}{k}n,
\]
we get that $H$ is $5k\r$-extremal, a contradiction. Thus, $H'$ is not $4k\r$-extremal and we can apply Lemma \ref{lem:alm_mat} on $H'$ with parameter $2\r$ and get a matching $M''$ in $H'$ that covers all but at most $k^2/(2\r)$ vertices. 
Since for every $(k+1)$-tuple $S$ in $V(H)$, the number of $S$-absorbing edges in $M'$ is at least $k/\r$, we can repeatedly absorb the leftover vertices (at most $k/(2\r)$ times, each time the number of leftover vertices is reduced by $k$) until the number of leftover vertices is at most $k$ (strictly less than $k$ by the assumption). Let $\tilde{M}$ denote the absorbing matching after the absorption. Then $\tilde{M}\cup M''$ is the desired near perfect matching in $H$.
\end{proof}

\section{Proof of Theorem \ref{thm:ext}}

We prove Theorem \ref{thm:ext} in this section. 
We use the following result of Pikhurko \cite{Pik}, stated here in a less general form.

\begin{theorem}\cite[Theorem 3]{Pik}\label{thm:pik}
Let $H$ be a $k$-partite $k$-graph with $k$-partition $V(H)=V_1\cup V_2\cup \cdots \cup V_k$ such that $|V_i|=m$ for all $i\in [k]$. Let $\delta_{\{1\}}(H)=\min \{ |N(v_1)|: v_1\in V_1 \}$ and
\[
\delta_{[k]\setminus \{1\}}(H) = \min\{ |N(v_2,\dots, v_k)|: v_i\in V_i \text{ for every }2\le i\le k \}.
\]
For sufficiently large integer $m$, if
\[
\delta_{\{1\}}(H) m + \delta_{[k]\setminus\{1\}}(H) m^{k-1} \ge \frac32 m^k,
\]
then $H$ contains a perfect matching.
\end{theorem}

\begin{proof}[Proof of Theorem \ref{thm:ext}]
Fix a sufficiently small $\e>0$. Suppose $n$ is sufficiently large and not divisible by $k$.
Let $H$ be a $k$-graph on $n$ vertices satisfying $\delta_{k-1}(H)\ge \lfloor\frac{n}{k}\rfloor$. Assume that $H$ is $\e$-extremal, namely, there is an independent set $S\subseteq V(H)$ with $|S| \ge (1-\e)\frac{k-1}k n$. 

We partition $V(H)$ as follows. Let $\a=\e^{1/2}$. Let $C$ be a maximum independent set of $V(H)$. Define
\begin{equation}\label{eq:A}
A=\left\{x\in V\setminus C: \deg(x, C)\ge (1-\a) \binom{|C|}{k-1}\right\},
\end{equation}
and $B=V\setminus (A\cup C)$. We first observe the following bounds of $|A|, |B|, |C|$.

\begin{proposition}\label{clm:size}
$|A|\ge\left\lfloor\frac{n}k\right\rfloor-\a n$, $|B|\le \a n$, and $(1-{\e})\frac {(k-1)n}k\le |C|\le \lceil \frac {(k-1)n}k\rceil$.
\end{proposition}

\begin{proof}
The lower bound for $|C|$ follows from our hypothesis immediately. For any $S\subseteq C$ of order $k-1$, we have $N(S)\subseteq A\cup B$. 
By the minimum degree condition, we have 
\begin{equation}\label{eq:ab}
\left\lfloor\frac{n}k\right\rfloor \le |N(S)| \le |A|+|B| =n-|C| \le \frac nk + \e \frac {(k-1)n}k,
\end{equation}
which gives the upper bound for $|C|$.
By the definitions of $A$ and $B$, we have
\[
\left\lfloor\frac{n}k\right\rfloor \binom{|C|}{k-1} \le e((A\cup B)C^{k-1})\le (1-\a)\binom{|C|}{k-1} |B| + \binom{|C|}{k-1} |A|,
\]
where $e((A\cup B)C^{k-1})$ denotes the number of edges that contains $k-1$ vertices in $C$ and one vertex in $A\cup B$.
Thus, we get $\left\lfloor\frac{n}k\right\rfloor \le |A|+|B|-\a |B|$, which gives that $\a |B| \le |A|+|B| - \left\lfloor\frac{n}k\right\rfloor\le \e n$ by \eqref{eq:ab}.
So $|B|\le \a n$ and $|A|\ge \left\lfloor\frac{n}k\right\rfloor-|B|\ge \left\lfloor\frac{n}k\right\rfloor-\a n$. 
\end{proof}

We will build four disjoint matchings $M_1$, $M_2$, $M_3$, and $M_4$ in $H$, whose union gives the desired near perfect matching in $H$. Let $r\equiv n \mod k$ and $1\le r\le k-1$. Note that $\lfloor \frac{n}{k}\rfloor=\frac{n-r}{k}$. 
For $i\in [3]$, let $A_i=A\setminus V(\cup_{j\in [i]} M_j)$ and $C_i=C\setminus V(\cup_{j\in [i]} M_j)$ be the sets of uncovered vertices of $A$ and $C$, respectively. Let $n_i=|V(H)\setminus V(\cup_{j\in [i]} M_j)|$ and note that $n_i\equiv r \mod k$.

\bigskip
\noindent \emph{Step 1. Small matchings $M_1$ and $M_2$ covering $B$.}

We build the first matching $M_1$ on vertices of $B\cup C$ of size $t$ only if $t:= \lfloor\frac{n}{k}\rfloor - |A|>0$. Note that it is possible that $t\le 0$ -- in this case $M_1=\emptyset$. By Proposition \ref{clm:size}, we know that $t=\lfloor\frac{n}{k}\rfloor - |A|\le \a n$. 
Since $\delta_{k-1}(H)\ge \lfloor\frac{n}{k}\rfloor$ and by the definition of $t$, we have $\delta_{k-1}(H[B\cup C])\ge t$. Since $|C|\le \lceil \frac {(k-1)n}k\rceil$, we have $|B|= n - |C| - |A|\ge \lfloor\frac{n}{k}\rfloor - |A|=t$. We pick arbitrary $t$ disjoint $(k-1)$-sets from $C$. Since $C$ is an independent set, each of the $(k-1)$-sets has at least $t$ neighbors in $B$, so we can choose a matching $M_1$ of size $t$.

Next we build the second matching $M_2$ that covers all the vertices in $B\setminus V(M_1)$. 
For each $v\in B\setminus V(M_1)$, we pick $k-2$ arbitrary vertices from $C$ not covered by the existing matching, and an uncovered vertex $v\in V$ to complete an edge and add it to $M_2$. Since $\delta_{k-1}(H)\ge \lfloor\frac{n}{k}\rfloor$ and the number of vertices covered by the existing matching is at most $k |B|\le k\a n<\lfloor\frac{n}{k}\rfloor$, such an edge always exists.

\medskip
Our construction guarantees that each edge in $M_1\cup M_2$ contains at least one vertex from $B$ and thus $|M_1\cup M_2|\le |B|$. 
We claim that $|A_1|\ge \frac{n_1-r}{k}$ and $|A_2|\ge \frac{n_2-r}{k}$.
To see the bound for $|A_1|$, we separate two cases depending on $t$. When $t>0$, since $|M_1|=t$, we have
\[
|A_1| = \frac{n-r}{k} - t = \frac{n - r - k|M_1|}{k}=\frac{n_1 - r}{k}.
\]
Otherwise $t\le 0$, we have $n_1=n$ and $|A_1|=|A|\ge \frac{n - r}{k}=\frac{n_1 - r}{k}$. 
For the bound for $|A_2|$, since each edge of $M_2$ contains at most one vertex of $A$, we have
\[
|A_2| \ge |A_1| - |M_2| \ge \frac{n_1 - r}{k} - |M_2| = \frac{n_2 - r}{k}.
\]

Let $s:=|A_2| - \frac{n_2-r}{k}\ge 0$. Since $n_2=n - k|M_1\cup M_2|\ge n - k|B|\ge n - k\a n$ and $|C|\ge (1-{\e})\frac {(k-1)n}k$ (Proposition \ref{clm:size}), we get
\[
s\le n - |C| - \frac{n-k\a n -r}{k}\le \e\frac {(k-1)n}k + \a n + 1 \le 2\a n.
\]

\bigskip
\noindent \emph{Step 2. A small matching $M_3$.}

Starting with $M_3=\emptyset$, we will greedily add at most $2\a n$ edges to $M_3$ from $A_2\cup C_2$ until we have $|A_3| - \frac{n_3 -r}{k}\in \{0, 1\}$. 
Indeed, throughout the process, denote by $n'$ the number of uncovered vertices of $H$ and denote by $A', C'$ the set of uncovered vertices in $A, C$, respectively. Let $c=|A'| - \frac{n' -r}{k}$.
If $c\ge k-1$, then we arbitrarily pick $k-1$ vertices from $A'$ and a vertex from $A'\cup C'$ to form an edge. As a result, $|A'| - \frac{n' -r}{k}$ decreases by $k-1$ or $k-2$. 
If $c<k-1$, then we pick $c$ vertices from $A'$, $k-c-1$ vertices from $C'$, and form an edge with some vertex from $A'\cup C'$. In this case, $|A'| - \frac{n' -r}{k}$ decreases by $c$ or $c-1$. The iteration stops when $|A'| - \frac{n' -r}{k}$ becomes 0 or 1 after at most $\lceil\frac{s}{k-2}\rceil\le s\le 2\a n$ steps.
Note that we can always form an edge in each step because the number of covered vertices is at most $k|B|+k\cdot 2\a n\le 3k\a n<\delta_{k-1}(H)$. So we get a matching $M_3$ of at most $2\a n$ edges.

\bigskip
\noindent \emph{Step 3. The last matching $M_4$.}

Now we have two cases, $|A_3| - \frac{n_3 -r}{k}= 0$ or 1. In the first case, we will find a matching $M_4$ of size $|A_3|$ which leaves $r$ vertices in $C_3$. In the second case, we will find a matching $M_4$ of size $|A_3| - 1$ which leaves one vertex in $A_3$ and $r-1$ vertices in $C_3$. 
Note that in either case we are done since $M = M_1\cup M_2\cup M_3\cup M_4$ is a matching that covers all but $r$ vertices of $V(H)$. 

We define $A_3'$ and $C_3'$ as follows. If $|A_3| - \frac{n_3 -r}{k}= 0$, we let $A_3' = A_3$ and obtain $C_3'$ by deleting arbitrary $r$ vertices from $C_3$. Otherwise, we obtain $A_3'$ by deleting one arbitrary vertex from $A_3$ and obtain $C_3'$ by deleting $r-1$ arbitrary vertices from $C_3$. Note that in both cases, we have $|A_3'| - \frac{|A_3'| + |C_3'|}{k}=0$, which implies $|C_3'| = (k-1)|A_3'|$. Furthermore, we have
\[
|A_3'|\ge |A| - |M_1\cup M_2| - |M_3| - 1\ge \left \lfloor \frac{n}{k} \right\rfloor - \a n -\a n - 2\a n -1 \ge \left \lfloor \frac{n}{k} \right\rfloor - 5\a n,
\]
because $|M_1\cup M_2|\le |B|\le \a n$ and $|M_3|\le 2\a n$.

Let $m:=|A_3'|$. Next, we partition $C_3'$ arbitrarily into $k-1$ parts $C^1, C^2,\dots, C^{k-1}$ of the same size $m$. We want to apply Theorem \ref{thm:pik} on the $k$-partite $k$-graph $H':=H[A_3', C^1,\dots, C^{k-1}]$. Let us verify the assumptions. 
First, since $C_3'$ is independent, for any set of $k-1$ vertices $v_1,\dots, v_{k-1}$ such that $v_i\in C^i$ for $i\in [k-1]$, the number of its non-neighbors in $A\cup B$ is at most
\[
|A| + |B| - \left\lfloor \frac{n}{k} \right\rfloor \le \frac nk + \e \frac {(k-1)n}k - \left\lfloor \frac{n}{k} \right\rfloor \le \e n\le 2k\e m,
\]
where we use \eqref{eq:ab} and the last inequality follows from $m=|A_3'|\ge \lfloor \frac{n}{k} \rfloor - 5\a n>\frac{k-1}{k^2}n$.
So we have $\delta_{[k]\setminus\{1\}}(H')\ge m - 2k\e m= (1-2k\e)m$.
Next, by \eqref{eq:A}, for any $v\in A_3'$, we have
\[
\overline{\deg}_{H}(v, C)\le \a \binom{|C|}{k-1}\le \a \frac{|C|^{k-1}}{(k-1)!}\le \a \frac{\left(\frac{k-1}{k}n \right)^{k-1}}{(k-1)!}\le \a \frac{(km)^{k-1}}{(k-1)!}= \a c_k m^{k-1},
\]
where $c_k=\frac{k^{k-1}}{(k-1)!}$.
This implies that $\delta_{\{1\}}(H') \ge (1-\a c_k) m^{k-1}$. Thus, we have
\[
\delta_{\{1\}}(H') m + \delta_{[k]\setminus\{1\}}(H') m^{k-1} \ge (1- \a c_k) m^{k-1} m + (1-2k \e)m m^{k-1}>\frac32 m^k,
\]
since $\e$ is small enough. By Theorem \ref{thm:pik}, we find a perfect matching in $H'$, which gives the perfect matching $M_4$ on $H[A_3'\cup C_3']$.
\end{proof}

\section{Acknowledgement}

The author would like to thank his advisor, Yi Zhao, for suggesting this problem. The author also thanks Andrzej Ruci\'nski, Yi Zhao and the anonymous referee for invaluable comments on the manuscript.

\bibliographystyle{plain}
\bibliography{Jan2014}

\end{document}